\documentclass[onefignum,onetabnum]{siamart190516}

\usepackage{algorithmicx}
\usepackage{algpseudocode}
\usepackage{lipsum}
\usepackage{amssymb}
\usepackage{amsmath}
\usepackage{amsfonts}
\usepackage{graphicx}
\usepackage{epstopdf}

\usepackage{amsopn}
\usepackage{microtype}
\usepackage{subfigure}

\ifpdf
  \DeclareGraphicsExtensions{.eps,.pdf,.png,.jpg}
\else
  \DeclareGraphicsExtensions{.eps}
\fi

\def\TITLE{An Improvement of the Resolvent Estimate in the Kreiss
Matrix Theorem}
\def\SHORTTITLE{Improvement of Resolvent Estimate}
\def\AUTHOR{Zeyu Jin}

\headers{\SHORTTITLE}{\AUTHOR}

\title{\TITLE}

\author{\AUTHOR\thanks{School of Mathematical Sciences, Peking 
University (\email{jinzy@pku.edu.cn}).}}

\newsiamremark{assumption}{Assumption}
\newsiamremark{example}{Example}
\newsiamremark{remark}{Remark}
\newsiamremark{experiment}{Experiment}
\newsiamremark{hypothesis}{Hypothesis}
\crefname{hypothesis}{Hypothesis}{Hypotheses}
\crefname{assumption}{Assumption}{Assumptions}
\newsiamthm{claim}{Claim}

\DeclareMathOperator{\diag}{diag}

\newcommand{\mc}{\mathbb{C}}
\newcommand{\mr}{\mathbb{R}}
\newcommand{\ms}{\mathcal{S}}
\newcommand{\mh}{\mathbb{H}}
\newcommand{\mk}{\mathcal{K}}
\newcommand{\mf}{\mathcal{F}}
\newcommand{\mn}{\mathbb{N}}
\newcommand{\mt}{\mathcal{T}}
\newcommand{\norm}[1]{\left\|{#1}\right\|}
\newcommand{\rd}{\mathrm{d}}
\renewcommand{\complement}{\mathsf{c}}

\ifpdf
\hypersetup{
  pdftitle={\TITLE},
  pdfauthor={\AUTHOR}
}
\fi

\usepackage{geometry}

\begin{document}

\maketitle

\begin{abstract}
  We improve the resolvent estimate in the Kreiss matrix theorem for 
  a set of matrices that generate uniformly bounded semigroups.
  The new resolvent estimate is proved to be equivalent to Kreiss's 
  resolvent condition, and it better describes the behavior of the 
  resolvents at infinity. 
\end{abstract}



\section{Introduction}
\label{sec:intro}

The Kreiss matrix theorem \cite{kreiss_uber_1959} is one of the  
fundamental results on the well-posedness for Cauchy problems in 
the theory of partial differential equations.
Among other things, the theorem asserts the equivalence of uniform 
boundedness of semigroups and a certain resolvent estimate for a set 
of matrices. There are many studies concerning the improvement of the 
uniform upper bounds of the semigroups of matrices under Kreiss's 
resolvent condition; see, e.g., 
\cite{leveque_resolvent_1984,van_linear_1993}. 
In this paper, we focus on the converse, that is, the improvement of 
the resolvent estimate in the Kreiss matrix theorem.
There have been several attempts in this aspect. For example, in 
\cite{miller_resolvent_1968}, Miller estimated the resolvents on 
certain contours in the left half-plane by classifying the spectrum 
of each matrix. In \cite{zarouf_sharpening_2009}, Zarouf improved the 
resolvent estimate for power-bounded matrices by using Bernstein-type 
inequalities for rational functions.

We notice that Kreiss's resolvent condition sometimes fails to give 
sharp estimates for well-posed Cauchy problems.
The main purpose of this paper is to establish a sharper resolvent 
estimate for a set of matrices that generate uniformly bounded 
semigroups.
This work is partially motivated by the following Cauchy problem:
\begin{equation}\label{equ:cauchy}
  \begin{aligned}
    \partial_t u & = \mathcal{A} u + f(x, t), \\
    u|_{t = 0} & = 0,
  \end{aligned}
\end{equation}
where $x \in \mr^n$ and $t \in [0, +\infty)$ are spatial and temporal 
variables, respectively; $u$ is the unknown $m$-vector function of 
$(x, t)$; the source term $f$ is an $m$-vector function of $(x, t)$;
the differential operator $\mathcal{A}$ is defined by 
$(\mathcal{A} v)^{\wedge}(\xi) := A(\xi) \hat{v}(\xi)$, where 
$\hat{v}$ is the Fourier transform of a function $v$; 
the symbol of $\mathcal{A}$, that is, $A(\xi)$, is an 
$m \times m$-matrix function of $\xi \in \mr^n$.
Under the assumption of well-posedness 
\cite{kreiss_initial-boundary_2004} of \eqref{equ:cauchy}, the Kreiss 
matrix theorem yields that there exist constants $\alpha \in \mr$ 
and $K > 0$ independent of $\xi$ such that 
\footnote{Recall that the $p$-norm of a given matrix 
$A \in \mc^{n \times n}$ is defined as 
\begin{displaymath}
  \norm{A}_p := \sup_{x \in \mc^n \setminus \left\{ 0 \right\}}
  \frac{\norm{A x}_p}{\norm{x}_p},
\end{displaymath}
where $\norm{x}_p$ is the $p$-norm of the vector $x$.
Throughout the paper, the norm $\norm{\cdot}_2$ is abbreviated as 
$\norm{\cdot}$ for both vectors and matrices.}
\begin{equation}\label{equ:cauchy_KMTbound}
  \norm{(z I - A(\xi))^{-1}} \le \frac{K}{\Re(z) - \alpha}, 
  \quad z \in \mc, \quad \Re(z) > \alpha.
\end{equation}
By applying Fourier transform and Laplace transform on both sides of 
\eqref{equ:cauchy}, one can obtain that 
\footnote{Assume that $\hat{v}(\xi, t)$ is the Fourier transform of 
$v(x, t)$. The Laplace transform of $\hat{v}$ can be defined as
\begin{displaymath}
  \tilde{v}(\xi, z) := \int_{0}^{+\infty} e^{-z t} \hat{v}(\xi, t) 
  \,\rd t,
\end{displaymath}
where $z$ is a complex number such that the integral is well-defined.}
\begin{displaymath}
  z \tilde{u}(\xi, z) = A(\xi) \tilde{u}(\xi, z) 
    + \tilde{f}(\xi, z),
\end{displaymath}
under some appropriate assumptions on the functions $A(\xi)$ and 
$f(x, t)$.
Therefore, one has that $\tilde{u}(\xi, z) = (z I - A(\xi))^{-1} 
\tilde{f}(\xi, z)$.
The inverse Laplace transform of $\tilde{u}(\xi, z)$ is determined 
by  
\begin{displaymath}
  \hat{u}(\xi, t) = \frac{1}{2 \pi i} 
    \int_{\gamma - i \infty}^{\gamma + i \infty}
    e^{z t} (z I - A(\xi))^{-1} \tilde{f}(\xi, z) \,\rd z,
\end{displaymath}
where $\gamma$ is a real number such that the contour line lies in 
the region of convergence of $\tilde{u}(\xi, z)$ for each 
$\xi \in \mr^n$, and the integral is understood in the sense of 
principal value.
One may estimate $\hat{u}(\xi, t)$ as 
\begin{displaymath}
  |\hat{u}(\xi, t)| \le \frac{e^{\gamma t}}{2 \pi} 
    \int_{- \infty}^{+ \infty} \norm{(z I - A(\xi))^{-1}}
    \cdot \norm{\tilde{f}(\xi, z)} \,\rd y,
\end{displaymath}
where $z = \gamma + i y$. The Kreiss matrix theorem gives an upper 
bound of $\norm{(z I - A(\xi))^{-1}}$ as in 
\eqref{equ:cauchy_KMTbound}, which is only dependent on the real part 
of $z$. However, this bound is too rough, since for fixed $\xi$ and 
$\gamma$, one has that 
\begin{displaymath}
  \norm{(z I - A(\xi))^{-1}} = \mathcal{O}(|y|^{-1}), 
  \quad \text{ as } |y| \rightarrow + \infty.
\end{displaymath}
To obtain a better upper bound of $\hat{u}(\xi, t)$, we need a sharper
estimate of the resolvent $(z I - A(\xi))^{-1}$.
Our new resolvent estimate better describes how the resolvents decay 
at infinity; see \Cref{thm:main_resolvent}.

The rest of this paper is organized as follows. In \Cref{sec:pre}, we 
recall the Kreiss matrix theorem given in 
\cite{kreiss_uber_1959,kreiss_initial-boundary_2004} and 
propose our main result, which is proved in \Cref{sec:proof}. 
In \Cref{sec:disccussion}, we make several remarks on the 
generalizations of our main theorem and its relationships with some 
previous work \cite{miller_resolvent_1968,zarouf_sharpening_2009}. 
The paper ends with a brief summary and conclusion in 
\Cref{sec:concl}.
\section{Preliminaries}
\label{sec:pre}

For the sake of simplicity, we introduce the following definitions.
Similar definitions can be found in \cite{Yong_singular_1992}.
\begin{definition}
  A matrix $M \in \mc^{n \times n}$ is called quasi-stable, if 
  \begin{displaymath}
    \sup_{t \ge 0} \norm{e^{M t}} < + \infty.
  \end{displaymath}
  A set of matrices $\mf \subset \mc^{n \times n}$ is 
  called uniformly quasi-stable, if 
  \begin{displaymath}
    \sup_{M \in \mf} \sup_{t \ge 0} \norm{e^{M t}} < + \infty.
  \end{displaymath}
\end{definition}

The Kreiss matrix theorem 
\cite{kreiss_uber_1959,kreiss_initial-boundary_2004} gives several 
necessary and sufficient conditions for the uniform quasi-stability 
of a set of matrices. The theorem is stated as follows.
\begin{theorem}[Kreiss matrix theorem. See Theorem 2.3.2 of 
  \cite{kreiss_initial-boundary_2004}]
  \label{thm:KMT}
  Let $\mf$ denote a set of matrices in $\mc^{n \times n}$.
  The following four conditions are equivalent.
  \begin{enumerate}
    \item There exists a constant $K_1$ such that  
      $\norm{e^{M t}} \le K_1$ for all $M \in \mf$ and 
      $t \ge 0$.
    \item There exists a constant $K_2$ such that 
      \begin{equation}\label{equ:KMT_resolvent}
        \norm{(z I - M)^{-1}} \le \frac{K_2}{\Re(z)}, 
      \end{equation}
      for all $M \in \mf$ and $\Re(z) > 0$.
    \item There exist constants $K_{31}$, $K_{32}$ such that 
      for each $M \in \mf$, there exists a 
      transformation $S = S(M)$ with 
      $\norm{S} + \norm{S^{-1}} \le K_{31}$,
      the matrix $S M S^{-1}$ is upper triangular,
      \begin{displaymath}
        S M S^{-1} = \begin{pmatrix}
          b_{11} & b_{12} & \cdots & b_{1n} \\
               & b_{22} & \cdots & b_{2n} \\
               &    & \ddots & \vdots \\
               &    &    & b_{nn} \\
        \end{pmatrix},
      \end{displaymath}
      the diagonal is ordered, 
      \begin{displaymath}
        0 \ge \Re(b_{11}) \ge \Re(b_{22}) \ge \cdots \ge 
        \Re(b_{nn}),
      \end{displaymath}
      and the upper diagonal elements satisfy the estimate 
      \begin{displaymath}
        |b_{ij}| \le K_{32} |\Re(b_{ii})|, \quad 
          1 \le i < j \le n.
      \end{displaymath}
    \item There exists a positive constant $K_4$ such that 
      for each $M \in \mf$, there exists a Hermitian 
      matrix $H = H(M)$ such that 
      \begin{displaymath}
        K_4^{-1} I \le H \le K_4 I, \quad H M + M^* H \le 0.
      \end{displaymath}
  \end{enumerate}
\end{theorem}

Hereinafter, we denote the open right half-plane as 
\begin{displaymath}
  \mh = \left\{ z \in \mc : \Re(z) > 0 \right\}.
\end{displaymath}
Given a matrix $M \in \mc^{n \times n}$, we denote the spectrum of
$M$ as $\sigma(M)$, and define 
\footnote{When $\sigma(M)$ is a subset of half-plane $\mh$, we define 
that $\mk(M) := + \infty$.}
\begin{displaymath}
  \mk(M) := \sup_{z \in \mh} \frac{\norm{(z I - M)^{-1}}}
  {\max_{\lambda \in \sigma(M) \setminus \mh} |z - \lambda|^{-1}},
\end{displaymath}
which turns out to be a new measurement of uniform quasi-stability.
Our main result is presented as follows.
\begin{theorem}\label{thm:main}
  Let $\mf$ denote a set of matrices in $\mc^{n \times n}$.
  The set of matrices $\mf$ is uniformly quasi-stable, if and only if 
  \begin{equation}\label{equ:uniformKM}
    \sup_{M \in \mf} \mk(M) < + \infty.
  \end{equation}
\end{theorem}

The following corollary of \Cref{thm:main} gives a sharper resolvent 
estimate under Kreiss's resolvent condition \eqref{equ:KMT_resolvent}.
The new resolvent estimate better describes the behavior of the 
resolvents at infinity.

\begin{corollary}\label{thm:main_resolvent}
  If a set of matrices $\mf$ is uniformly quasi-stable, then there 
  exists a positive constant $K$ such that 
  \begin{equation}\label{equ:main_resolvent}
    \norm{(z I - M)^{-1}} \le K \max_{\lambda \in \sigma(M)} 
    |z - \lambda|^{-1},
  \end{equation}
  for all $M \in \mf$ and $z \in \mh$.
\end{corollary}

\section{Proof of Main Result}
\label{sec:proof}

Let us first recall the following characterization of quasi-stability 
given in \cite{kreiss_initial-boundary_2004}.
\begin{lemma}[See Lemma 2.3.1 of \cite{kreiss_initial-boundary_2004}]
  \label{lemma:QS_character}
  For any $M \in \mc^{n \times n}$, the following two conditions are 
  equivalent.
  \begin{enumerate}
    \item The matrix $M$ is quasi-stable. 
    \item All eigenvalues $\lambda$ of the matrix $M$ have a real 
      part $\Re(\lambda) \le 0$. Furthermore, if $J_r$ is a Jordan 
      block of the Jordan matrix $J = S M S^{-1}$ which corresponds 
      to an eigenvalue $\lambda_r$ of the matrix $M$ with 
      $\Re(\lambda_r) = 0$, then $J_r$ has dimension $1 \times 1$.
  \end{enumerate}
\end{lemma}

Using \Cref{lemma:QS_character}, we can prove that the quantity 
$\mk(M)$ is actually a measurement of quasi-stability of a 
given matrix $M$.
\begin{lemma}\label{lemma:QS_KM}
  Let $M \in \mc^{n \times n}$. The matrix $M$ is quasi-stable, if 
  and only if $\mk(M) < + \infty$.
\end{lemma}
\begin{proof}
  First we assume that the matrix $M$ is quasi-stable. 
  \Cref{lemma:QS_character} yields that the spectrum $\sigma(M)$ is 
  a subset of the closed left half-plane $\mh^\complement$.
  Let $J = S M S^{-1}$ be the Jordan matrix of $M$. We notice that 
  \begin{displaymath}
    \norm{(z I - M)^{-1}} \le \norm{S^{-1}} \cdot 
    \norm{(z I - J)^{-1}} \cdot \norm{S}.
  \end{displaymath}
  To estimate $\norm{(z I - J)^{-1}}$, it suffices to estimate the 
  resolvent of each Jordan block. Let 
  \begin{displaymath}
    J_r = \begin{pmatrix}
      \lambda_r & 1 & & \\
      & \ddots & \ddots & \\
      & & \ddots & 1 \\
      & & & \lambda_r
    \end{pmatrix} = \lambda_r I + N \in \mc^{k \times k}
  \end{displaymath}
  be an arbitrary Jordan block of $J$.
  By direct calculation, one has that 
  \begin{equation}\label{equ:resolventJr}
    (z I - J_r)^{-1} = \sum_{j = 0}^{k - 1} 
      (z - \lambda_r)^{-(j + 1)} N^j, \quad z \in \mh.
  \end{equation}
  If $\Re(\lambda_r) < 0$, then for each $z \in \mh$, one has that 
  \begin{displaymath}
    \begin{aligned}
      \norm{(z I - J_r)^{-1}} & \le \left( 
        \sum_{j = 0}^{k - 1} |z - \lambda_r|^{-j} \norm{N}^j \right)
        \cdot |z - \lambda_r|^{-1} \\
      & \le \left( \sum_{j = 0}^{k - 1} |\Re(\lambda_r)|^{-j} 
        \norm{N}^j \right) \cdot 
        \max_{\lambda \in \sigma(M) \setminus \mh} |z - \lambda|^{-1}.
    \end{aligned}
  \end{displaymath}
  If $\Re(\lambda_r) = 0$, then \Cref{lemma:QS_character} yields that 
  the order of $J_r$ is $k = 1$, and thus 
  \begin{displaymath}
    \norm{(z I - J_r)^{-1}} = |z - \lambda_r|^{-1} \le 
    \max_{\lambda \in \sigma(M) \setminus \mh} |z - \lambda|^{-1},
    \quad z \in \mh.
  \end{displaymath}
  Therefore, we obtain that $\mk(M) < + \infty$.

  Conversely, we assume that $\mk(M) < + \infty$. Let $\lambda_r$ be 
  an arbitrary eigenvalue of $M$, and $J_r \in \mc^{k \times k}$ be 
  the largest Jordan block corresponding to $\lambda_r$. By 
  \eqref{equ:resolventJr}, one can obtain that 
  \begin{displaymath}
    \norm{(z I - M)^{-1}} \sim |z - \lambda_r|^{-k}, \quad 
    \text{ as } z \rightarrow \lambda_r.
  \end{displaymath}
  If $\Re(\lambda_r) > 0$, then 
  \begin{displaymath}
    \frac{\norm{(z I - M)^{-1}}}
    {\max_{\lambda \in \sigma(M) \setminus \mh} |z - \lambda|^{-1}}
    \sim |z - \lambda_r|^{-k}, \quad \text{ as } 
    z \rightarrow \lambda_r.
  \end{displaymath}
  If $\Re(\lambda_r) = 0$ and $k > 1$, then 
  \begin{displaymath}
    \frac{\norm{(z I - M)^{-1}}}
    {\max_{\lambda \in \sigma(M) \setminus \mh} |z - \lambda|^{-1}}
    \sim |z - \lambda_r|^{- k + 1}, \quad \text{ as }
    z \rightarrow \lambda_r.
  \end{displaymath}
  Both cases yield a contradiction, and then the quasi-stability of 
  the matrix $M$ follows from \Cref{lemma:QS_character}.
\end{proof}

Here is an elementary fact from linear algebra. It asserts that 
for a set of unit upper triangular matrices
\footnote{Recall that a unit upper triangular matrix is an upper 
triangular matrix with $1$ on the diagonal.}
in $\mc^{n \times n}$, uniform boundedness is equivalent to uniform 
boundedness of inverse.
\begin{lemma}\label{lemma:unittri}
  If the matrix $A \in \mc^{n \times n}$ is a unit upper triangular 
  matrix with $\norm{A} \le \alpha$, then one has that 
  $\norm{A^{-1}} \le (n \alpha)^{n - 1}$.
\end{lemma}
\begin{proof}
  Let the column partitioning of the strictly upper triangular 
  component of $A$ be 
  \begin{displaymath}
    U := A - I = \begin{pmatrix}
      0 & u_2 & \cdots & u_n
    \end{pmatrix}.
  \end{displaymath}
  Denote the $i$th column of the identity matrix $I$ as $e_i$.
  It is clear that $e_i^\top u_j = 0$ for $2 \le j \le i \le n$.
  Therefore, one can obtain that 
  \begin{displaymath}
    A = I + u_n e_n^\top + u_{n - 1} e_{n - 1}^\top + 
    \ldots + u_2 e_2^\top = \left( I + u_n e_n^\top \right) 
    \left( I + u_{n - 1} e_{n - 1}^\top \right)
    \ldots \left( I + u_2 e_2^\top \right),
  \end{displaymath}
  and that 
  \begin{displaymath}
    I = \left( I + u_j e_j^\top \right) 
    \left( I - u_j e_j^\top \right), \quad 
    \text{ for } j = 2, 3, \ldots, n,
  \end{displaymath}
  which yields that 
  \begin{displaymath}
    A^{-1} = \left( I - u_2 e_2^\top \right)
    \left( I - u_3 e_3^\top \right) \ldots 
    \left( I - u_n e_n^\top \right).
  \end{displaymath}
  By the equivalence of $1$-norm and $2$-norm for matrices, one can 
  obtain that
  \begin{displaymath}
    \norm{I - u_j e_j^\top} \le \sqrt{n} \norm{I - u_j e_j^\top}_1 
    \le \sqrt{n} \norm{A}_1 \le n \norm{A} \le n \alpha, 
  \end{displaymath}
  for each $j = 2, 3, \ldots, n$.
  Therefore, one can obtain that 
  \begin{displaymath}
    \norm{A^{-1}} \le \norm{I - u_2 e_2^\top}
    \norm{I - u_3 e_3^\top} \ldots 
    \norm{I - u_n e_n^\top} \le (n \alpha)^{n - 1},
  \end{displaymath}
  which completes the proof.
\end{proof}

We are now set to prove our main result.
\begin{proof}[Proof of \cref{thm:main}]
  First we assume that the set of matrices $\mf$ satisfies 
  \eqref{equ:uniformKM}. By \Cref{lemma:QS_KM}, each matrix in $\mf$ 
  is quasi-stable. For each $M \in \mf$, $z \in \mh$, and 
  $\lambda \in \sigma(M) = \sigma(M) \setminus \mh$, one has that 
  \begin{displaymath}
    0 < \Re(z) \le \Re(z - \lambda) \le |z - \lambda|,
  \end{displaymath}
  which yields that 
  \begin{displaymath}
    \sup_{M \in \mf} \sup_{z \in \mh} \left( \Re(z) \cdot 
    \norm{(z I - M)^{-1}} \right) \le \sup_{M \in \mf} 
    \sup_{z \in \mh} \frac{\norm{(z I - M)^{-1}}}
    {\max_{\lambda \in \sigma(M) \setminus \mh} |z - \lambda|^{-1}}
    = \sup_{M \in \mf} \mk(M) < + \infty.
  \end{displaymath}
  Therefore, the set of matrices $\mf$ is uniformly quasi-stable 
  thanks to \Cref{thm:KMT}.

  Conversely, we assume that $\mf$ is uniformly quasi-stable. At this 
  time, the spectrum $\sigma(M)$ is a subset of the closed left 
  half-plane $\mh^\complement$ for each $M \in \mf$. 
  Using the notations in the third 
  condition of \Cref{thm:KMT}, we denote the diagonal component and 
  the strictly upper triangular component of $S M S^{-1}$ by $D$ and 
  $N$, respectively, that is, 
  \begin{displaymath}
    D = D(M) = \diag \left\{ b_{11}, b_{22}, \ldots, b_{nn} \right\},
    \quad N = N(M) = S M S^{-1} - D.
  \end{displaymath}
  By direct calculation, for any $z \in \mh$, one has that 
  \begin{displaymath}
    (z I - M)^{-1} = S^{-1} (z I - D - N)^{-1} S = 
    S^{-1} \left( I - (z I - D)^{-1} N \right)^{-1} (z I - D)^{-1} S,
  \end{displaymath}
  which yields that 
  \begin{displaymath}
    \norm{(z I - M)^{-1}} \le \kappa(S) \cdot 
    \norm{\left( I - (z I - D)^{-1} N \right)^{-1}} \cdot 
    \norm{(z I - D)^{-1}},
  \end{displaymath}
  where the condition number of $S$ can be estimated as 
  \begin{displaymath}
    \kappa(S) := \norm{S} \cdot \norm{S^{-1}} \le 
    \frac{\left( \norm{S} + \norm{S^{-1}} \right)^2}{4} \le 
    \frac{K_{31}^2}{4}.
  \end{displaymath}
  For each $1 \le i < j \le n$, the absolute value of the $(i,j)$-th 
  entry of $(z I - D)^{-1} N$ is 
  \begin{equation}\label{equ:entry_estimate}
    \frac{|b_{ij}|}{|z - b_{ii}|} \le 
    \frac{K_{32} |\Re(b_{ii})|}{|\Re(z) - \Re(b_{ii})|} \le K_{32}.
  \end{equation}
  \Cref{lemma:unittri} yields that there exists a constant $C > 0$ 
  dependent on $n$ and $K_{32}$ such that 
  \begin{displaymath}
    \norm{\left( I - (z I - D)^{-1} N \right)^{-1}} \le C,
  \end{displaymath}
  since the matrix $I - (z I - D)^{-1} N$ is unit upper triangular.
  We notice that 
  \begin{displaymath}
    \norm{(z I - D)^{-1}} = \max_{\lambda \in \sigma(M)} 
    |z - \lambda|^{-1} = 
    \max_{\lambda \in \sigma(M) \setminus \mh} 
    |z - \lambda|^{-1}.
  \end{displaymath}
  Therefore, we have that 
  \begin{displaymath}
    \sup_{M \in \mf} \mk(M) \le \frac{K_{31}^2 C}{4} < + \infty,
  \end{displaymath}
  which completes the proof of the theorem.
\end{proof}
\section{Discussions}
\label{sec:disccussion}
In this section, we generalize our main theorem and discuss its 
relationships with several previous results.

\subsection{Resolvent estimate in the resolvent set}
\label{sec:remark_miller}
From the proof of the necessity part of \Cref{thm:main}, one can 
observe that it is not essential to take $z$ in the open right 
half-plane $\mh$ to obtain a resolvent estimate as in 
\eqref{equ:main_resolvent}. Given a matrix $M \in \mf$ and a number 
$r > 0$, let us denote 
\begin{displaymath}
  \ms(M, r) := 
  \left\{ z \in \mc : \max_{\lambda \in \sigma(M)} 
  \frac{|\Re(\lambda)|}{|z - \lambda|} \le \frac{1}{r} \right\}.
\end{displaymath}
The following theorem provides a resolvent estimate at any point in 
the resolvent set. Similar results can be found in 
\cite{miller_resolvent_1968}.
\begin{theorem}\label{thm:main_miller}
  If a set of matrices $\mf$ is uniformly quasi-stable, then there 
  exists a positive constant $K$ such that 
  \begin{displaymath}
    \norm{(z I - M)^{-1}} \le K \left( 1 + \frac{1}{r} \right)^{n - 1}
    \max_{\lambda \in \sigma(M)} |z - \lambda|^{-1},
  \end{displaymath}
  for all $M \in \mf$, $r > 0$ and $z \in \ms(M, r)$.
\end{theorem}
\begin{proof}
  Let us take $M \in \mf$, $r > 0$ and $z \in \ms(M, r)$. Using the 
  notations in the third condition of \Cref{thm:KMT}, for each 
  $1 \le i < j \le n$, one can obtain that 
  \begin{displaymath}
    \frac{|b_{ij}|}{|z - b_{ii}|} \le 
    \frac{K_{32} |\Re(b_{ii})|}{|z - b_{ii}|} \le \frac{K_{32}}{r},
  \end{displaymath}
  which is analogous to \eqref{equ:entry_estimate}. 
  According to \Cref{lemma:unittri}, it follows that there exists a 
  constant $C$ dependent on $n$ and $K_{32}$ such that 
  \begin{displaymath}
    \norm{\left( I - (z I - D)^{-1} N \right)^{-1}} \le 
    C \left( 1 + \frac{1}{r} \right)^{n - 1}.
  \end{displaymath}
  Hence the proof is completed by a similar argument in the proof of 
  \Cref{thm:main}.
\end{proof}

\subsection{Power-bounded matrices}
Analogous results for power-bounded matrices can be proved in a 
similar way. Let us recall the following Kreiss matrix theorem 
\cite{kreiss_stabilitatsdefinition_1962} for power-bounded matrices.
\begin{theorem}[Kreiss matrix theorem. See Satz 4.1 of 
  \cite{kreiss_stabilitatsdefinition_1962}]
  \label{thm:KMT_power}
Let $\mf$ denote a set of matrices in $\mc^{n \times n}$.
The following four conditions are equivalent.
\begin{enumerate}
  \item There exists a constant $K_1$ such that  
    $\norm{M^\nu} \le K_1$ for all $M \in \mf$ and 
    $\nu \in \mn$.
  \item There exists a constant $K_2$ such that 
    \begin{displaymath}
      \norm{(z I - M)^{-1}} \le \frac{K_2}{|z| - 1}, 
    \end{displaymath}
    for all $M \in \mf$ and $z \in \mc$ with $|z| > 1$.
  \item There exist constants $K_{31}$, $K_{32}$ such that 
    for each $M \in \mf$, there exists a 
    transformation $S = S(M)$ with 
    $\norm{S} + \norm{S^{-1}} \le K_{31}$,
    the matrix $S M S^{-1}$ is upper triangular,
    \begin{displaymath}
      S M S^{-1} = \begin{pmatrix}
        b_{11} & b_{12} & \cdots & b_{1n} \\
              & b_{22} & \cdots & b_{2n} \\
              &    & \ddots & \vdots \\
              &    &    & b_{nn} \\
      \end{pmatrix},
    \end{displaymath}
    the diagonal is ordered, 
    \begin{displaymath}
      1 \ge |b_{11}| \ge |b_{22}| \ge \cdots \ge 
      |b_{nn}|,
    \end{displaymath}
    and the upper diagonal elements satisfy the estimate 
    \begin{displaymath}
      |b_{ij}| \le K_{32} (1 - |b_{ii}|), \quad 
        1 \le i < j \le n.
    \end{displaymath}
  \item There exists a positive constant $K_4$ such that 
    for each $M \in \mf$, there exists a Hermitian 
    matrix $H = H(M)$ such that 
    \begin{displaymath}
      K_4^{-1} I \le H \le K_4 I, \quad M^* H M \le H.
    \end{displaymath}
\end{enumerate}
\end{theorem}

Similar to \Cref{sec:remark_miller}, we would like to give a 
resolvent estimate for a set of uniformly power-bounded matrices 
$\mf$ at any point in the resolvent set. Given a matrix $M \in \mf$ 
and a number $r > 0$, let us denote 
\begin{displaymath}
  \mt(M, r) := 
  \left\{ z \in \mc : \max_{\lambda \in \sigma(M)} 
  \frac{1 - |\lambda|}{|z - \lambda|} \le \frac{1}{r} \right\}.
\end{displaymath}
We have the following theorem. The second conclusion of 
\Cref{thm:main_power} can be found in \cite{zarouf_sharpening_2009}.
\begin{theorem}\label{thm:main_power}
  If a set of matrices $\mf$ satisfies any condition of 
  \Cref{thm:KMT_power}, then there exists a positive constant $K$ 
  such that 
  \begin{displaymath}
    \norm{(z I - M)^{-1}} \le K \left( 1 + \frac{1}{r} \right)^{n - 1}
    \max_{\lambda \in \sigma(M)} |z - \lambda|^{-1},
  \end{displaymath}
  for all $M \in \mf$, $r > 0$ and $z \in \mt(M, r)$. In particular, 
  one has that 
  \begin{displaymath}
    \sup_{M \in \mf} \sup_{|z| > 1} 
    \frac{\norm{(z I - M)^{-1}}}{\max_{\lambda \in \sigma(M)} 
    |z - \lambda|^{-1}} < + \infty.
  \end{displaymath}
\end{theorem}

\begin{proof}
  Let us take $M \in \mf$, $r > 0$ and $z \in \mt(M, r)$. 
  Using the notations in the third condition of \Cref{thm:KMT_power}, 
  for each $1 \le i < j \le n$, one can obtain that 
  \begin{displaymath}
    \frac{|b_{ij}|}{|z - b_{ii}|} \le 
    \frac{K_{32} (1 - |b_{ii}|)}{|z - b_{ii}|} \le \frac{K_{32}}{r}.
  \end{displaymath}
  The remainder of the proof of the first conclusion is analogous to 
  that in \Cref{thm:main_miller}. The second conclusion follows from 
  an observation that 
  \begin{displaymath}
    |z - \lambda| \ge |z| - |\lambda| > 1 - |\lambda| \ge 0,
  \end{displaymath}
  for each $\lambda \in \sigma(M)$ and $z \in \mc$ with $|z| > 1$, 
  and thus $z \in \mt(M, 1)$.
\end{proof}

\section{Conclusions}
\label{sec:concl}
We establish a sharper resolvent estimate under the same conditions 
as those of the Kreiss matrix theorem. 
The new resolvent estimate can be used to derive sharper estimates 
for well-posed Cauchy problems. 
As future work, it is unknown how the sharpest constant $K$ in 
the estimate \eqref{equ:main_resolvent} depends on the dimension $n$. 
It would also be interesting to know how the results in this 
paper can be generalized to certain classes of linear operators on 
Hilbert space. 

\section*{Acknowledgments}
The author is greatly indebted to Prof. Ruo Li for his wholehearted 
guidance and support and for his valuable suggestions on the 
presentation of this paper, and to Dr. Yizhou Zhou for many helpful 
and insightful discussions and for reading an earlier version of this 
paper.

\bibliographystyle{siamplain}
\bibliography{ref}

\end{document}